\documentclass[reqno,12pt]{amsart}
\usepackage[english]{babel}
\usepackage{amsmath,amsfonts,amsthm,amssymb,amsbsy,upref,color,graphicx,hyperref,enumerate,comment}
\usepackage[a4paper,margin=2.75truecm]{geometry}

\DeclareMathOperator{\diam}{diam}
\def\step#1#2{\par\noindent{\underline{\it Step~#1.}}\emph{ #2}\\}

\def\eps{{\varepsilon}}
\def\N{\mathbb{N}}
\def\O{\Omega}
\def\R{\mathbb{R}}
\def\A{\mathcal{A}}
\def\HH{\mathcal{H}}
\newcommand{\be}{\begin{equation}}
\newcommand{\ee}{\end{equation}}
\newcommand{\bib}[4]{\bibitem{#1}{\sc#2: }{\it#3. }{#4.}}
\def\res{\mathop{\hbox{\vrule height 7pt width .5pt depth 0pt \vrule height .5pt width 6pt depth 0pt}}\nolimits}

\numberwithin{equation}{section}
\theoremstyle{plain}

\newtheorem{theo}{Theorem}[section]
\newtheorem{lemm}[theo]{Lemma}

\newtheorem{prop}[theo]{Proposition}
\newtheorem{defi}[theo]{Definition}
\newtheorem{conj}[theo]{Conjecture}

\title[On the relations between principal eigenvalue and torsional rigidity]{On the relations between principal eigenvalue and torsional rigidity}

\author[M. van den Berg]{Michiel van den Berg}

\author[G. Buttazzo]{Giuseppe Buttazzo}

\author[A. Pratelli]{Aldo Pratelli}

\date{29 October 2019}

\begin{document}

\begin{abstract}
We consider the problem of minimising or maximising the quantity $\lambda(\O)T^q(\O)$ on the class of open sets of prescribed Lebesgue measure. Here $q>0$ is fixed, $\lambda(\O)$ denotes the first eigenvalue of the Dirichlet Laplacian on $H^1_0(\O)$, while $T(\O)$ is the torsional rigidity of $\O$. The optimisation problem above is considered in the class of {\it all domains} $\O$, in the class of {\it convex domains} $\O$, and in the class of {\it thin domains}. The full Blaschke-Santal\'o diagram for $\lambda(\O)$ and $T(\O)$ is obtained in dimension one, while for higher dimensions we provide some bounds.
\end{abstract}

\maketitle



\section{Introduction\label{sintro}}

In this paper we consider the problem of minimising or maximising the quantity
\[
\lambda^\alpha(\O)T^\beta(\O)
\]
on the class of open sets $\O\subset\R^d$ having a prescribed Lebesgue measure $|\O|$ with $0<|\O|<\infty$. Here $T(\O)$ is the torsional rigidity of $\O$, defined by
\[
T(\O)=\int_\O w_\O\,dx\,,
\]
where $w_\O$ is the unique solution of the Dirichlet problem
\be\label{pde}
\begin{cases}
-\Delta w=1&\hbox{in }\O\,,\\
w\in H^1_0(\O)\,,
\end{cases}
\ee
and $\lambda(\O)$ is the first eigenvalue of the Dirichlet Laplacian $-\Delta$ on $H^1_0(\O)$. That is the minimal value $\lambda$ such that the PDE
\be
\label{pde2}\begin{cases}
-\Delta u=\lambda u&\hbox{in }\O\,,\\
u\in H^1_0(\O),
\end{cases}
\ee
has a non-zero solution. By the min-max principle (see e.g.~\cite{he06}) we also have that
\[\lambda(\O)=\min\bigg\{\Big[\int_\O|\nabla u|^2\,dx\Big]\Big[\int_\O u^2\,dx\Big]^{-1}\ :\ u\in H^1_0(\O),\ u\ne0\bigg\}.\]
Throughout this paper we adopt the following notation. If $\Omega$ is open in $\R^d$ with $0<|\O|<\infty$ then $\Omega^*$ is a ball in $\R^d$ with $|\O^*|=|\O|$. Furthermore $B_R$ is a ball with radius $R$. We put $\omega_d=|B_1|$.

The case $\beta=0$ is well known: the Faber-Krahn inequality (see for instance~\cite{he06}, \cite{hepi05}) asserts that
\[\lambda(\O^*)\le\lambda(\O)\,.\]
We also have that
\[\sup\Big\{\lambda(\O)\ :\O\, \textup {open in}\, \R^d,\, |\O|=\omega_d\Big\}=+\infty\,.\]
Indeed, we obtain a lower bound for the supremum above by choosing for $\O$ the disjoint union of $n$ balls with measure $\omega_d/n$ each. This gives
\[\sup\Big\{\lambda(\O)\ :\O\, \textup {open in}\, \R^d,\, |\O|=1\Big\}\ge n^{2/d}\lambda(B_1)\,,\]
where we have used the scaling relation
\be\label{scaling1}
\lambda(t\O)=t^{-2}\lambda(\O),\, t>0\,,
\ee
and the observation that if $\O$ is the a disjoint union of a family of open sets $\O_{\gamma},\, \gamma \in \Gamma$ then $\lambda(\O)=\inf_{\gamma\in\Gamma} \lambda(\O_{\gamma})$.

Similarly, the case $\alpha=0$ can be solved by a symmetrisation argument (see for instance~\cite{he06}, \cite{hepi05}), which gives the Saint-Venant inequality,
\be\label{SVI}
T(\O)\le T(\O^*)\,.
\ee
We also have that
\[\inf\Big\{T(\O)\ :\O\textup{ open in }\R^d,\ |\O|=1\Big\}=0\,.\]
Indeed, we obtain an upper bound for the infimum above by choosing for $\O$ the disjoint union of $n$ balls with measure $\omega_d/n$ each. This gives
\[\inf\Big\{T(\O)\ :\O\,\textup{open in }\, \R^d,\, |\O|=1\Big\}\le n^{-(d+2)/d}T(B_{1})\,,\]
where we have used the scaling relation
\be\label{scaling2}
T(t\O)=t^{d+2}T(\O),\,t>0\,,
\ee
and the observation that if $\O$ is the a disjoint union of a family of open sets $\O_{\gamma},\, \gamma \in \Gamma$ then $T(\O)=\sum_{\gamma\in\Gamma} T(\O_{\gamma})$.

Note that by~\eqref{SVI} and~\eqref{scaling2} we have
\be\label{SVI2}
\frac{T(\O)}{|\O|^{(d+2)/2}}\le \frac{T(B)}{|B|^{(d+2)/2}}\,,
\ee
where $B$ is any ball.

The case when $\alpha$ and $\beta$ have a different sign is also easy: by the inequalities above we obtain for $\alpha>0$ and $\beta<0$
\[\begin{split}
&\min\Big\{\lambda^\alpha(\O)T^\beta(\O)\ :\ |\O|=\omega_d\Big\}=\lambda^\alpha(B_1)T^\beta(B_1)\,,\\
&\sup\Big\{\lambda^\alpha(\O)T^\beta(\O)\ :\ |\O|=\omega_d\Big\}=+\infty\,,
\end{split}\]
while for $\alpha<0$ and $\beta>0$
\[\begin{cases}
&\inf\Big\{\lambda^\alpha(\O)T^\beta(\O)\ :\ |\O|=\omega_d\Big\}=0\\
&\max\Big\{\lambda^\alpha(\O)T^\beta(\O)\ :\ |\O|=\omega_d\Big\}=\lambda^\alpha(B_1)T^\beta(B_1)\,.
\end{cases}\]

It remains to consider the case $\alpha>0$ and $\beta>0$. Setting $q=\beta/\alpha>0$ we can limit ourselves to deal with the quantity
\[\lambda(\O)T^q(\O)\,.\]

Using~\eqref{scaling1} and~\eqref{scaling2} we can remove the constraint of prescribed Lebesgue measure on $\O$ by normalising the quantity $\lambda(\O)T^q(\O)$, and multiply it by a suitable power of $|\O|$. We then end up with the scaling invariant shape functional
\[F_q(\O)=\frac{\lambda(\O)T^q(\O)}{|\O|^{(dq+2q-2)/d}}\]
that we want to minimise or maximise over the class of open sets $\O\subset\R^d$ with $0<|\O|<\infty$.

We recall that the Sobolev space $H^1_0(\O)$ can also be defined for quasi open sets $\O$, and that~\eqref{pde} and~\eqref{pde2} admit solutions $w_\O$ and $u_{\O}$ respectively. The solution $w_{\O}$ is unique. Hence, the torsional rigidity $T(\O)$ is defined for every bounded {\it quasi open set}. It is well-known that the boundedness of $\O$ is not necessary to have a finite value of $T(\O)$, for which the assumption that $\O$ is of finite Lebesgue measure is enough.
Since eigenvalue $\lambda(\O)$, torsional rigidity $T(\O)$ and measure $|\O|$ can be defined for every quasi open  set $\O$ (see for instance~\cite{bubu05}), the functional $F_q$ is defined on the class of all quasi open subsets $\O$ of $\R^d$. More generally, the eigenvalue $\lambda(\mu)$ and the torsional rigidity $T(\mu)$ can be defined for every {\it capacitary measure} $\mu$. Hence we may define $F_q(\mu)$ on the class of capacitary measures (see the Appendix).

The inequalities above for the functionals $F_q,\, q>0$ provide some bounds for the study of the Blaschke-Santal\'o diagram for $\lambda(\O)$ and $T(\O)$. That diagram identifies the subset $E$ of $\R^2$ whose coordinates are determined by $\lambda(\O)$ and $T(\O)$. We study this issue in Section~\ref{sdiag}, where we normalise the coordinates to vary in the interval $[0,1]$ (see~\eqref{defxy}). We obtain the full description of the Blaschke-Santal\'o diagram only for $d=1$, while for $d>1$ we only provide some bounds. Further properties of the Blaschke-Santal\'o diagram for $\lambda(\O)$ and $T(\O)$ are investigated in~\cite{luzu19}.

\section{Preliminaries\label{spreli}}

Define the torsion energy by
\[
E(\O)=\min\left\{\int_\O\Big(\frac12|\nabla u|^2-u\Big)\,dx\ :\ u\in H^1_0(\O)\right\}.
\]
We see easily that~\eqref{pde} is the Euler-Lagrange equation for $E(\O)$, and that
\[
T(\O)=-2E(\O)=\max\left\{\int_\O\Big(-|\nabla u|^2+2u\Big)\,dx\ :\ u\in H^1_0(\O)\right\}\,.
\]
By considering $tu,\,t>0$ instead of $u$ in the maximisation above, and by optimising with respect to $t$ we obtain the alternative formula
\[
T(\O)=\max\bigg\{\Big[\int_\O u\,dx\Big]^2\Big[\int_\O|\nabla u|^2\,dx\Big]^{-1}\ :\ u\in H^1_0(\O),\ u\ne0\bigg\}\,.
\]

The torsional rigidity of a ball can be easily computed in polar coordinates: if $B_R$ is centred at the origin, then the solution $w_{B_R}$ of~\eqref{pde} is given by
\[w_{B_R}(x)=\frac{R^2-|x|^2}{2d}\,.\]
Hence
\begin{equation}\label{star}
T(B_R)=\frac{\omega_d}{d(d+2)}\,R^{d+2}\,.
\end{equation}
Similarly for a ball $B_R$ centred at the origin
\[\lambda(B_R)=\frac{j^2_{d/2-1}}{R^{2}}\]
where $j_{d/2-1}$ is the first positive zero of the Bessel function $J_{d/2-1}$. The corresponding eigenfunction is given by $u_{B_R}(x)=J_{(d-2)/2}(j_{(d-2)/2}|x|/R),\, x\in B_R$. For example,
\[\lambda(B_1)=5.783\dots\qquad\hbox{if }d=2\,.\]

We focus now on some estimates for the torsional rigidity of cylinders. We consider cylinders of the form
\[R_{A,h}=A\times]-h/2,h/2[\,,\]
where $A$ is an open set in $\R^{d-1}$ with finite Lebesgue measure, and $h>0$. We denote by $d_A(x)$ the distance of a point $x\in A$ from $\partial A$, and by $A_q$ the set
\[A_q=\big\{x\in A\ :\ d_A(x)>q\big\}\,.\]
We also denote by $|\partial A|$ the $\HH^{d-2}$ measure of $\partial A$, and by $|A|$ the $\HH^{d-1}$ measure of $A$. The closure of $A$ is denoted by $\overline{A}$. We recall Definition 6.1 in~\cite{vdBD89}.
\begin{defi}
An open set $A$ has $R$-smooth boundary if for all $x_0\in \partial A$ there exist two open balls $B_R(x_1),\,B_R(x_2)$ with radii $R$ such that (i) $B_R(x_1)\subset A,\, B_R(x_2)\subset \R^{d-1}\setminus \overline{A}, \overline{B_R(x_1)}\cap\overline{B_R(x_2)}=\{x_0\}$, (ii) the previous inclusions do not hold for $R$ replaced by  $R+\varepsilon$ for any $\varepsilon>0$.
\end{defi}
If $A$ is a bounded set with $C^2$ boundary, $\partial A$ is also $C^{1,1}$. We infer by Lemma 2.2 in~\cite{AKSZ} that $\partial A$ is $R$-smooth for some $R>0$. An important preliminary result is the following.
\begin{theo}\label{2.2}
If $A$ is open in $\R^{d-1}$ with finite Lebesgue measure, then
\begin{equation}\label{t1}
T(R_{A,h})\le \frac{|A|h^3}{12}\,.
\end{equation}
If $A$ is open, bounded and convex, then
\begin{equation}\label{t2}
T(R_{A,h})\ge \frac{|A|h^3}{12}- \frac{31\cdot2^{(d-4)/2}\zeta(5)}{\pi^4}|\partial A|h^4\,.
\end{equation}
If $A$ is open, bounded with $C^2$ boundary $\partial A$ then
\begin{equation}\label{t3}
\bigg|T(R_{A,h})-\frac{|A|h^3}{12}+\frac{31\zeta(5)}{4\pi^5}|\partial A|h^4\bigg|\le \frac{10^{d-2}|A|h^5}{12R^2}\,.
\end{equation}
\end{theo}
\begin{proof}
The heat equation proof below preserves the Cartesian product structure of $A\times]0,h[$ up to the very last step which consists of an integration over $t$. The setup is as follows.
Let $\Omega$ be an open set in $\R^d$, with finite Lebesgue measure, and let $u_{\Omega}$ be the solution of
\[
\Delta u=\frac{\partial u}{\partial t},\, \textup {on}\, \Omega\,,
\]
with initial condition
\[
\lim_{t\downarrow 0}u_{\Omega}(x;t)=1, \, x\in \Omega\,,
\]
and $u_{\Omega}(\cdot;t)\in H_0^1(\Omega).$ It is straightforward to verify that
\[
w_{\Omega}(x)=\int_{[0,\infty)}dt\,u_{\Omega}(x;t)\,,
\]
and, by Tonelli's theorem,
\[
T(\Omega)=\int_{[0,\infty)}dt\int_{\Omega}dx\,u_{\Omega}(x;t)\,.
\]
For $A$ open, and $h>0$,
\begin{equation}\label{t8}
u_{R_{A,h}}(x;t)=u_A(x';t)u_{]0,h[}(x_1;t),\, (x_1,x')\in R_{A,h}\,,
\end{equation}
with obvious notation. The solution $u_{]0,h[}$ is given in terms of the $L^2(]0,h[)$ spectral resolution of the Dirichlet Laplacian on $]0,h[$,
\[
u_{]0,h[}(x_1;t)=\frac{2}{h}\sum_{k=1}^{\infty}e^{-t\pi^2k^2/h^2}\sin\bigg(\frac{\pi kx_1}{h}\bigg)\int_{]0,h[}dy_1\,\sin\bigg(\frac{\pi ky_1}{h}\bigg)\,.
\]
By the maximum principle, or by probabilistic tools, one can show (Corollary 6.4 in~\cite{vdBD89}) that for any open set $A\subset \R^{d-1}$,
\begin{equation}\label{t10}
1\ge u_A(x';t)\ge 1- 2^{(d+1)/2}e^{-d_A(x')^2/(8t)}\,.
\end{equation}
To prove the assertion under~\eqref{t1}, we have by the first inequality in~\eqref{t10}, Tonelli's theorem, and the positivity of $u_{]0,h[}$,
\begin{equation}\begin{split}\label{t11}
T(R_{A,h})&\le \int_{[0,\infty)}dt \int_A dx'\,\int_{]0,h[}dx_1\,u_{]0,h[}(x_1;t)\\
&=\int_{[0,\infty)}dt\,\sum_{k=1,3,\dots}e^{-t\pi^2k^2/h^2}\int_A dx'\,\frac{2}{h}\bigg(\int_{]0,h[}dx_1\,\sin\bigg(\frac{\pi kx_1}{h}\bigg)\bigg)^2\\
&=\sum_{k=1,3,\dots}\frac{8h^3|A|}{\pi^4k^4}
=\frac{15}{16}\sum_{k\in \N}\frac{8h^3|A|}{\pi^4k^4}
=\frac{|A|h^3}{12}\,,
\end{split}\end{equation}
where we have used that $\zeta(4)=\pi^4/90$.

To prove the assertion under~\eqref{t2} it suffices to bound the contribution of the second term in the right-hand side of~\eqref{t10} to $T(R_{A,h})$ from above.
We have by the coarea formula that for convex bounded $A$, and $f(d_Ax)\ge 0$,
\[
\int_Adx'f(d_A(x'))\le |\partial A|\int_{[0,\infty)}dq\,f(q)\,,
\]
where we have used that $|\partial A_q|\le |\partial A|$. See Proposition 2.4.3 in~\cite{bubu05}. Hence
\begin{equation}\label{t13}\begin{split}
\int_{[0,\infty)}&dt\, \int_A dx'\int_{]0,h[}dx_1\, 2^{(d+1)/2}e^{-d_A(x')^2/(8t)}u_{]0,h[}(x_1;t)\\
&\leq |\partial A|\int_{[0,\infty)}dt\,\int_{]0,h[}dx_1\,\int_{[0,\infty)}dq\,2^{(d+1)/2}e^{-q^2/(8t)}u_{]0,h[}(x_1;t)\\
&=2^{(d+1)/2}|\partial A|\int_{[0,\infty)}dt\,(2\pi t)^{1/2}\sum_{k\in \N}e^{-t\pi^2k^2/h^2}\frac{2}{h}\bigg(\int_{]0,h[}dx_1\,\sin\bigg(\frac{\pi kx_1}{h}\bigg)\bigg)^2\\
&=\sum_{k=1,3,\dots}\frac{2^{(d+6)/2}|\partial A|h^4}{\pi^4k^5}
=\frac{31\cdot2^{(d-4)/2}\zeta(5)}{\pi^4}|\partial A|h^4\,.
\end{split}\end{equation}
This, together with the calculation under~\eqref{t11}, proves~\eqref{t2}.

To prove the assertion under~\eqref{t3} we use the first part of Theorem 6.2 in~\cite{vdBD89} which reads that for $A$ open, bounded with $C^2$ and $R$-smooth boundary,
\begin{equation}\label{t14}
\bigg|\int_Adx'\,u_A(x';t)-|A|+\frac{2|\partial A|t^{1/2}}{\pi^{1/2}}\bigg|\le \frac{10^{d-1}|A|t}{R^2},\,t>0\,.
\end{equation}
Multiplying both sides of~\eqref{t14} with $\int_{]0,h[}dx_1u_{]0,h[}(x_1;t)$ gives, by~\eqref{t8},
\begin{align}\label{t15}
\bigg|\int_{R_{A,h}}dx\,u_{R_{A,h}}(x;t)-&|A|\sum_{k=1,3,\dots}\frac{8h}{\pi^2k^2}e^{-t\pi^2k^2/h^2}+\frac{2|\partial A|t^{1/2}}{\pi^{1/2}}\sum_{k=1,3,\dots}\frac{8h}{\pi^2k^2}e^{-t\pi^2k^2/h^2}\bigg|\nonumber \\ &\le \frac{10^{d-1}|A|t}{R^2}\sum_{k=1,3,\dots}\frac{8h}{\pi^2k^2}e^{-t\pi^2k^2/h^2},\,t>0\,,
\end{align}
where have used
\[
\int_{]0,h[}dx_1u_{]0,h[}(x_1;t)=\sum_{k=1,3,\dots}\frac{8h}{\pi^2k^2}e^{-t\pi^2k^2}\,.
\]
We complete the proof by integrating~\eqref{t15} with respect to $t$ over $[0,\infty)$. This gives, using  $\zeta(6)=\pi^6/945$,
\begin{gather*}
\int_{[0,\infty)}dt\,\sum_{k=1,3,\dots}\frac{8h}{\pi^2k^2}e^{-t\pi^2k^2/h^2}=\frac{h^3}{12}\,,\\
\int_{[0,\infty)}dt\,\sum_{k=1,3,\dots}\frac{8h}{\pi^2k^2}te^{-t\pi^2k^2/h^2}=\frac{h^5}{120}\,,\\
\int_{[0,\infty)}dt\,\sum_{k=1,3,\dots}\frac{8h}{\pi^2k^2}t^{1/2}e^{-t\pi^2k^2/h^2}=\frac{31\zeta(5)h^4}{4\pi^5}\,.
\end{gather*}
The first and the third formula above were also used in~\eqref{t11} and~\eqref{t13} respectively.
\end{proof}

Proposition 3.2 of~\cite{bbv15} asserts that for a rectangle $R_{L,h}$ of with sides of length $L$ and $H$ respectively
\[
\left|T(R_{L,h})-\frac{h^3L}{12}+\frac{31\zeta(5)h^4}{2\pi^5}\right| \le\frac{h^5}{15L}\,.
\]
This jibes with~\eqref{t3} since $|\partial ]0,L[|=2$.

The eigenvalues of a cylinder $R_{A,h}$ are easily computed by separation of variables. For example,
\[
\lambda(R_{A,h})=\frac{\pi^2}{h^2}+\lambda(A)\,.
\]
In particular, for $d$-rectangles
\[\O=\prod_{k=1}^d]0,L_k[\]
we have
\be\label{lrect}
\lambda\Big(\prod_{k=1}^d]0,L_k[\Big)=\pi^2\sum_{k=1}^d\frac{1}{L_k^2}\,.
\ee

\section{General domains\label{sgene}}

A first case to consider is when $0<q\le2/(d+2)$. The Kohler-Jobin result (see~\cite{kj78a,kj78b,bra14} for a survey and some generalizations) states that
\be\label{kjineq}
\lambda(B)T^{2/(d+2)}(B)\le\lambda(\O)T^{2/(d+2)}(\O)\qquad\hbox{for every }\O\subset\R^d,
\ee
and is crucial to provide a lower bound to $F_q$.

\begin{prop}
If $0<q\le2/(d+2)$, then
\[\min\Big\{F_q(\O)\ :\ \O\textup{ open in }\R^d,\ |\O|<\infty\Big\}=F_q(B)\]
where $B$ is any ball in $\R^d$.
\end{prop}

\begin{proof}
It is enough to write
\[
F_q(\O)=\lambda(\O)T^{2/(d+2)}(\O)\frac{T^{q-2/(d+2)}(\O)}{|\O|^{(dq+2q-2)/d}}
=\lambda(\O)T^{2/(d+2)}(\O)\Big[\frac{T(\O)}{|\O|^{(d+2)/d}}\Big]^{q-2/(d+2)}\,,
\]
and to apply the Kohler-Jobin inequality~\eqref{kjineq} together with the fact that $q\le2/(d+2)$, and that the quantity $T(\O)|\O|^{-(d+2)/d}$ is maximal when $\O$ is a ball.
\end{proof}

If $q>2/(d+2)$ then the infimum of $F_q$ is zero, as shown below.

\begin{prop}\label{finf}
If $q>2/(d+2)$, then
\[
\inf\Big\{F_q(\O)\ :\ \O\textup{ open in }\R^d,\ |\O|<\infty\Big\}=0\,.
\]
\end{prop}

\begin{proof}
Let $\O$ be the disjoint union of $B_1$ and $N$ disjoint balls of radius $\eps\in (0,1]$. Then we have
\[
F_q(\O)=\frac{\lambda(B)\big(T(B)+N\eps^{d+2}T(B)\big)^q}{\big(|B|+N\eps^d|B|\big)^{(dq+2q-2)/d}}
=F_q(B)\,\frac{(1+N\eps^{d+2})^q}{(1+N\eps^d)^{(dq+2q-2)/d}}\,.
\]
Taking now $N\in \N$ such that $\eps^{-d-2}\le N<\eps^{-d-2}+1$ gives
\[
F_q(\O)\le F_q(B)\frac{3^q}{(1+\eps^{-2})^{(dq+2q-2)/d}}\,,
\]
which vanishes as $\eps\to0$ since the exponent in the denominator is positive.
\end{proof}

We now deal with the supremum of $F_q$ for $0<q<1$.

\begin{prop}
Let $0<q<1$. Then
\[\sup\Big\{F_q(\O)\ :\ \O\textup{ open in }\R^d,\ |\O|<\infty\Big\}=+\infty.\]
\end{prop}

\begin{proof}
Let $\O$ be a $d$-rectangle of sides $L_k$ ($k=1,\dots,d$) and take $L_1=\eps$, and $L_k=1$ for $k\ge2$. Then, by Theorem~\ref{2.2} and~\eqref{lrect} we have
\[
F_q(\O)\approx\frac{\pi^2\big(\eps^{-2}+(d-1)\big)(\eps^3/12)^q}{\eps^{(dq+2q-2)/d}}=\frac{\pi^2\big(1+\eps^2(d-1)\big)}{12^q\,\eps^{2(1-q)(d-1)/d}}\,,
\]
which diverges to $+\infty$ as $\eps\to0$ since $q<1$.
\end{proof}

We consider now the case $q>1$. By Proposition~\ref{finf} we have
\[
\inf\Big\{F_q(\O)\ :\ \O\,\textup{open in}\subset\R^d,|\O|<\infty\Big\}=0\,.
\]
Below we show that the supremum is finite.

\begin{prop}
If $q>1$, then
\[
\sup\Big\{F_q(\O)\ :\ \O\textup{ open in }\R^d,\ |\O|<\infty\Big\}\le\Big(\frac{1}{d(d+2)\omega_d^{2/d}}\Big)^{q-1}\,.
\]
\end{prop}

\begin{proof}
It is enough to apply the inequality (see Proposition 2.3 of~\cite{bbv15} or Theorem 1.1 of~\cite{bfnt16})
\[
\frac{\lambda(\O)T(\O)}{|\O|}\leq 1,
\]
to get
\[
F_q(\O)=\frac{\lambda(\O)T(\O)}{|\O|}\Big(\frac{T(\O)}{|\O|^{(d+2)/d}}\Big)^{q-1}\le\Big(\frac{T(\O)}{|\O|^{(d+2)/d}}\Big)^{q-1}\,.
\]
The conclusion now follows by~\eqref{SVI2} and~(\ref{star}).
\end{proof}

We do not know the exact value of the supremum of $F_q$ in the proposition above, and whether this supremum is attained.

Finally, the case $q=1$ was considered in~\cite{bfnt16}. There it was shown that
\begin{equation}\label{donethere}
\sup\Big\{F_1(\O)\ :\ \O\textup{ open in }\R^d,\ |\O|<\infty\Big\}=1.
\end{equation}
In the Appendix we provide an independent and shorter proof.

We may collect the estimates about general domains in Table~\ref{tablegeneral}.
\begin{table}[ht]
\centering
\begin{tabular}{c|l|c}
          		& General domains $\O$	&\\
\hline
				&						&\\
$0<q\le2/(d+2)$	&$\min F_q(\O)=F_q(B)$	&$\sup F_q(\O)=+\infty$\\
				&						&\\
\hline
				&						&\\
$2/(d+2)<q<1$	&$\inf F_q(\O)=0$		&$\sup F_q(\O)=+\infty$\\
				&						&\\
\hline
				&						&\\
$q=1$			&$\inf F_q(\O)=0$		&$\sup F_q(\O)=1$\\
				&						&\\
\hline
				&						&\\
$q>1$			&$\inf F_q(\O)=0$		&$\sup F_q(\O)<+\infty$\\
				&						&\\
\hline
\end{tabular}
\caption{Bounds for $F_q(\O)$ when $\O$ varies among all domains.}\label{tablegeneral}
\end{table}

\section{Convex domains\label{sconv}}

In the case of convex domains, some of the bounds seen in Section~\ref{sgene} remain: taking as $\O$ a slab $A\times]-\eps/2,\eps/2[$ we obtain
\begin{align*}
\inf\Big\{F_q(\O)\ :\ \O\,\hbox{bounded, convex, and open in $\R^d$}\Big\}=0\,, &&\hbox{if }q>1\,,
\end{align*}
and
\begin{align*}
\sup\Big\{F_q(\O)\ :\ \O\,\hbox{bounded, convex, and open in $\R^d$}\Big\}=+\infty\,, && \hbox{if }0<q<1\,.
\end{align*}

The case $q=1$ was studied in~\cite{bfnt16}, where the following bounds have been obtained:
\be\label{bounds}\begin{aligned}
\inf\Big\{F_1(\O)\ :\ \O\ \hbox{bounded, convex, and open in $\R^d$}\Big\}&=C_d^->0\,,\\
\sup\Big\{F_1(\O)\ :\ \O\ \hbox{bounded, convex, and open in $\R^d$}\Big\}&=C_d^+<1\,.
\end{aligned}\ee

The other cases follow easily from the bounds above.

\begin{prop}\label{sharpv}
We have
\[
\inf\Big\{F_q(\O)\ :\ \O\,\hbox{bounded, convex, and open in $\R^d$}\Big\}\ge C_d^-\big(d(d+2)\omega_d^{2/d}\big)^{1-q}
\]
if $q<1$, while
\[
\sup\Big\{F_q(\O)\ :\ \O\,\hbox{bounded, convex, and open in $\R^d$}\Big\}\le C_d^+\big(d(d+2)\omega_d^{2/d}\big)^{1-q}
\]
if $q>1$.
\end{prop}

\begin{proof}
Since
\[
F_q(\O)=F_1(\O)\Big(\frac{T(\O)}{|\O|^{1+2/d}}\Big)^{q-1}\,,
\]
it is enough to apply the bounds~\eqref{bounds} to get for $\O$ bounded, convex and open in $\R^d$
\begin{align*}
F_q(\O)&\geq C_d^-\Big(\frac{T(\O)}{|\O|^{1+2/d}}\Big)^{q-1}\ge C_d^-\Big(\frac{T(B)}{|B|^{1+2/d}}\Big)^{q-1} &\hbox{if }q<1\,,\\
F_q(\O)&\leq C_d^+\Big(\frac{T(\O)}{|\O|^{1+2/d}}\Big)^{q-1}\le C_d^+\Big(\frac{T(B)}{|B|^{1+2/d}}\Big)^{q-1} &\hbox{if }q>1\,,
\end{align*}
where $B$ is any ball. Since
\[
\frac{T(B)}{|B|^{1+2/d}}=\frac{1}{d(d+2)\omega_d^{2/d}}\,,
\]
the proposition follows.
\end{proof}

The explicit values of $C_d^-$ and $C_d^+$ for the case $q=1$ are not yet known. Looking at the results for thin domains in Section~\ref{sthin} and Corollary 1.6 in~\cite{bfnt19} we make the following conjecture.

\begin{conj}\label{conje}
The optimal values $C_d^+$ and $C_d^-$ in Proposition~\ref{sharpv} are given by
\begin{align*}
C_d^+=\frac{\pi^2}{12}\,, &&
C_d^-=\frac{\pi^2}{12}\frac{6}{(d+1)(d+2)}\,.
\end{align*}
The constant $C_d^+$ is asymptotically reached by a thin ``slab'', $\O_\eps=A\times[0,\eps]$, where $A$ is any open, bounded, convex $(d-1)$- dimensional set, and $\eps\to0$.
The constant $C_d^-$ is asymptotically reached by a thin ``cone set'' in the sense of Definition~\ref{defconefct}.
\end{conj}

The conjecture for $C_2^-$ is supported by the recent results in~\cite{bfnt19} where it is shown that if $\O$ is an isosceles triangles then $F_1(\O)\ge \frac{\pi^2}{24}$, and that this value is sharp in the limit where the quotient of height and base of the isosceles triangle becomes small.
\medskip

The question of existence of optimal convex domains for the shape functional $F_q$ arises. We will now prove the existence of a convex minimiser when $0<q<1$ and of a convex maximiser when $q>1$, while the existence for the case $q=1$ is open. Throughout we denote for a non-empty open bounded set $\O$ its {\it inradius} by
\be\label{e1}
r(\O)=\sup\big\{d_{\Omega}(x)\ :\ x\in\O\big\}\,,
\ee
and its {\it diameter} by
\[
\diam(\O)=\sup\big\{|x_1-x_2|\ :\ x_1\in\O,\ x_2\in\O\big\}\,.
\]

\begin{theo}\label{the1}
Let $q>1$. Then the shape optimisation problem
\[
\max\big\{F_q(\O)\ :\ \O\text{ open, bounded, convex in }\R^d\big\}
\]
has a maximiser $\O^+$dependent on $d$ and on $q$. Furthermore
\be\label{e4}
\frac{r(\O^+)}{\diam(\O^+)}\geq \frac{\omega_{d-1}\pi^d}{d\omega_d2^d}\,\big(d(d+2)\big)^{dq/(2(1-q))}\big(j_{(d-2)/2}\big)^{d/(q-1)}\,.
\ee
\end{theo}

\begin{theo}\label{the2}
Let $0<q<1$. Then the shape optimisation problem
\[
\min\big\{F_q(\O)\ :\ \O\text{ open, bounded, convex in }\R^d\big\}
\]
has a minimiser $\O^-$ dependent on $d$ and on $q$. Furthermore
\be\label{e6}
\frac{r(\O^-)}{\diam(\O^-)}\ge\pi2^{(5q-4)/(2(1-q))}\big(j_{(d-2)/2}\big)^{1/(q-1)}\,.
\ee
\end{theo}

\begin{proof}[Proof of Theorem~\ref{the1}] Let $q>1$. Since $F_q(t\O)=F_q(\O)$ for every $t>0$, we can consider a maximising sequence $\Omega_n$ for $F_q(\Omega)$ with $r(\O_n)$ fixed. If the diameter of $\O_n$ is uniformly bounded in $n$, then there exists a sequence of translates of a subsequence $(\O_{n_k})$, which converges in both the Hausdorff metric and complementary Hausdorff metric to say $\O^*$. Moreover torsional rigidity, principal Dirichlet eigenvalue and measure are continuous in both these metrics on the class of open, bounded, convex sets. See reference~\cite{hepi05}. To obtain an upper bound on the diameter, we use the fact that $\lambda(\Omega)T(\Omega)\leq |\Omega|$, as seen in~(\ref{donethere}), to obtain that
\be\label{e8}
F_q(\O)\le \frac{\lambda(\O)^{1-q}}{|\O|^{2(q-1)/d}}\,.
\ee
By~\cite{PS} we have for an open, bounded, convex set $\O$,
\be\label{e9}
\lambda(\O)\ge\frac{\pi^2}{4r(\O)^2}\,.
\ee
Let $0$ be a point at which the distance function in~\eqref{e1} has a maximum. Let $d_1$ and $d_2$ be two points of $\partial\O$ such that
\[
|d_1-d_2|=\diam(\O)\,.
\]
The $(d-1)$-dimensional plane perpendicular to $d_1-d_2$ intersects $B_{r(\O)}(0)$ in a $(d-1)$-dimensional disc with radius $r(\O)$. The union of the two cones having this disc as base and with vertices $d_1$ and $d_2$ has volume $(d-1)\omega_{d-1} \diam(\O)r(\O)^{d-1}$, and since these cones are contained in $\O$ we deduce
\[
|\O|\ge\frac{(d-1)\omega_{d-1}}{d}\diam(\O)r(\O)^{d-1}\,.
\]
This estimate, recalling~\eqref{e8} and~\eqref{e9}, gives
\[
\frac{r(\O)}{\diam(\O)}\ge\frac{\omega_{d-1}\pi^d}{d2^d}F_q(\O)^{d/(2(q-1))}\,.
\]
Since for any element $\O_n$ of a maximising sequence we have $F_q(\O_n)\ge F_q(B)$, this gives
\be\label{e13}
\frac{r(\O_n)}{\diam(\O_n)}\ge\frac{\omega_{d-1}\pi^d}{d2^d}F_q(B)^{d/(2(q-1))}\,.
\ee
Since $r(\O_n)$ is fixed, \eqref{e13} gives the required uniform upper bound for $\diam(\O_n)$. A straightforward computation shows that
\be\label{e14}
F_q(B)=\big(j_{(d-2)/2}\big)^2(d(d+2))^{-q}\omega_d^{2(1-q)/d}\,,
\ee
so~\eqref{e4} follows by~\eqref{e13} and~\eqref{e14}.
\end{proof}

\begin{proof}[Proof of Theorem~\ref{the2}]
Let $0<q<1.$ We follow the same strategy as in the proof of Theorem~\ref{the1}, and fix the inradius of the elements of a minimising sequence. To obtain a uniform upper bound on the diameter we proceed as follows. For an open, bounded, convex set in $\R^d$ we have by Theorem 1.1(i) in~\cite{DPGGLB} in the special case $p=q=2$ that
\[
\frac{T(\O)}{|\O|M(\O)}\ge \frac{2}{d(d+2)}\,,
\]
where $M(\O)$ is the maximum of the torsion function. On the other hand it is well known that $M(\O)\ge\lambda(\O)^{-1}$, see for example~\cite{vdB} and the references therein. It follows that
\[
F_q(\O)\ge
\bigg(\frac{2}{d(d+2)}\bigg)^q\frac{\lambda(\O)^{1-q}}{|\O|^{2(q-1)/d}}\,,
\]
which by~\eqref{e9} implies
\[
F_q(\O)\ge \frac{2^{3q-2}\pi^{2(1-q)}}{\big(d(d+2)\big)^q}\frac{r(\O)^{2(q-1)}}{|\O|^{2(q-1)/d}}\,.
\]
Furthermore, by the isodiametric inequality (see for instance~\cite{GE}),
\[
|\O|\le\frac{\omega_d}{2^d}\diam(\O)^d\,.
\]
The last two estimates, together with the fact that $F_q(B)\geq F_q(\Omega_n)$ for elements of a sequence minimizing $F_q$, imply
\be\label{e19}
\frac{r(\O_n)}{\diam(\O_n)}\ge\pi 2^{(5q-4)/(2(1-q))}\omega_d^{1/d}(d(d+2)^{q/(2(q-1))}F_q(B)^{1/(2(q-1))}\,.
\ee
Since $r(\O_n)$ is fixed, then $\diam(\O_n)$ is uniformly bounded from above. This completes the proof of the existence of a minimiser, and the estimate~\eqref{e6} directly comes by putting together~\eqref{e14} and~\eqref{e19}.
\end{proof}

We may then summarize the results about the case of convex domains in Table~\ref{tableconvex}.

\begin{table}[ht]
\centering
\begin{tabular}{c|l|c}
          		& Convex domains $\O$	&\\
\hline
				&						&\\
$q<1$			&$\min F_q(\O)>0$		&$\sup F_q(\O)=+\infty$\\
				&						&\\
\hline
				&						&\\
$q=1$			&$\inf F_1(\O)=C_d^->0$	&$\sup F_1(\O)=C_d^+<1$\\
				&						&\\
\hline
				&						&\\
$q>1$			&$\inf F_q(\O)=0$		&$\max F_q(\O)<+\infty$\\
				&						&\\
\hline
\end{tabular}
\medskip\caption{Bounds for $F_q(\O)$ when $\O$ varies among convex domains.}\label{tableconvex}
\end{table}

\section{Thin domains\label{sthin}}

In this section we analyse the case $q=1$ when $\O_\eps$ is a {\it thin domain}. More precisely, we consider
\[\O_\eps=\big\{(s,t)\ :\ s\in A,\ \eps h_-(s)<t<\eps h_+(s)\big\}\]
where $\eps$ is a small positive parameter, $A$ is a (smooth) domain of $\R^{d-1}$, and $h_-,h_+$ are two given (smooth) functions. We denote by $h(s)$ the {\it local thickness}
\[h(s)=h_+(s)-h_-(s),\]
and we assume that $h(s)\ge0$. The asymptotics for $\lambda(\Omega_\eps)$ and $T(\Omega_\eps)$ have been obtained in~\cite{bofr10}  and~\cite{bofr13}, and their first terms are
\begin{align*}
\lambda(\O_\eps)\approx\frac{\eps^{-2}\pi^2}{\|h\|^2_{L^\infty(A)}}\,, &&
T(\O_\eps)\approx\frac{\eps^3}{12}\int_A h^3(s)\,ds\,,
\end{align*}
which together give the asymptotic formula
\[
F_1(\O_\eps)\approx\frac{\pi^2}{12}\Big[\int_A h^3(s)\,ds\Big]\Big[\|h\|^2_{L^\infty(A)}\int_A h\,ds\Big]^{-1}\,.
\]
We now consider the case of {\it convex thin domains}, where the set $A\subset\R^{d-1}$ is convex, and the function $h:A\to\R^+$ is concave. In this case, we will see that the maximal and minimal possible values for $\lim_{\eps\to 0} F_1(\O_\eps)$ are respectively reached by the constant functions and by the cone functions, in the sense below. Different types of estimates for integrals involving powers of concave functions have been obtained in~\cite{kazi75}.

\begin{defi}\label{defconefct}
Let $A\subset \R^{d-1}$ be a convex set of positive measure, and let $P$ be an internal point of $A$. We call \emph{cone function} the smallest concave function $h:A\to [0,1]$ such that $h(P)=1$ and $h\res \partial A=0$. Notice that the level sets of $h$ are all homothetic copies of $A$. More precisely, for every $0\leq \sigma\leq 1$ the level set $\{h\geq \sigma\}$ is given by the set $\sigma P + (1-\sigma) A$. The set $\Omega_\eps$ is correspondingly called a \emph{cone set}.
\end{defi}

\begin{prop}\label{pconvthin}
Let $A\subset\R^{d-1}$ be a convex set. Then for every concave function $h:A\to\R^+$ with $\|h\|_{L^\infty(A)}=1$ we have
\be\label{thingendim}
\frac 6{(d+1)(d+2)}\le\frac{\begin{aligned}\int_A h^3(x)\,dx\end{aligned}}{\begin{aligned}\int_A h(x)\,dx\end{aligned}}\le1\,.
\ee
Moreover, both inequalities are sharp. In particular, for any $(d-1)$-dimensional convex set $A$, the right inequality is an equality if and only $h\equiv1$, while the left inequality is an equality if and only if $h$ is a cone function.
\end{prop}

\begin{proof}
Since the proof is quite involved, we divide it in several steps.

\medskip\step{I}{Preliminary notation.}
We start by considering the simpler case of a radial function, the general case will be studied only at the last step. Hence, from now and until the last step of the proof, we assume that $A$ is a $(d-1)$-dimensional ball and we consider a radially symmetric, decreasing, concave function $h:A\to [0,1]$ with $h(0)=1$. With minor abuse of notation we will denote by $h$ also the $1$-dimensional shape of $h$, that is, we write $h(x)=h(|x|)$. In particular, there is some $M>0$ such that $h:[0,M]\to [0,1]$, with $h(0)=1$, and the left inequality in~\eqref{thingendim} (observe that the other one is trivial)
 can be rewritten as
\be\label{goal}
\int_0^M \big(h^3(s)-C_d h(s)\big) s^{d-2}\,ds\ge0\,,
\ee
where we denote for brevity by $C_d$ the constant in the left of~\eqref{thingendim}, so for instance $C_2=1/2$ and $C_3=3/10$, and $C_d<1$ for every dimension. Notice that there is no need to consider the case $M=\infty$, because this corresponds to a function $h$ which is constantly $1$, for which the result is obvious. We will also call
\[H=\max\big\{s\ :\ h(s)=1\big\}\,,\qquad K=\max\big\{s\le M\ :\ h(s)\ge\sqrt{C_d}\big\}\,.\]

\medskip\step{II}{If $K=M$, then~\eqref{goal} holds strictly.}
In this very short step, we only observe that the case $K=M$ is not interesting. In fact, if $K=M$ this means that $h(s)\ge\sqrt{C_d}$ for every $0\le s\le M$, thus the integrand in~\eqref{goal} is pointwise strictly positive, except possibly at the sole $s=M$, so the validity of~\eqref{goal} is trivial and there is no equality case. As a consequence, from now on we assume $K<M$, so in particular $h(K)=\sqrt{C_d}$. Notice that the integrand in~\eqref{goal} is positive in the interval $[0,K]$ and negative in the interval $[K,M]$.

\medskip\step{III}{There are no flat parts, i.e., $H=0$.}
This step is devoted to reduce ourselves to the case when $h$ is strictly decreasing, so $H=0$. To be precise, assume for a moment that $H>0$, and call $\tilde h:[0,M-H]\to\R^+$ the concave function given by $\tilde h(s)=h(s+H)$. We want to show that if
\be\label{assmpt}
\int_0^M \big(h^3(s)-C_d h(s)\big)s^{d-2}\,ds\le0\,,
\ee
then
\be\label{thesis}
\int_0^{M-H} \big(\tilde h^3(s)-C_d\tilde h(s)\big)s^{d-2}\,ds<0\,.
\ee
To do so, we start by observing that, being $H>0$, then~\eqref{assmpt} implies
\be\label{removed}
\int_H^M \big(h^3(s)-C_d h(s)\big)s^{d-2}\,ds<0\,.
\ee
Recalling again that $h^3(s)-C_d h(s)$ is positive in $[0,K]$ and negative in $[K,M]$, we deduce that
\[\begin{split}
\int_0^{K-H} \big(\tilde h^3(s)- C_d \tilde h(s)\big)s^{d-2}\,ds
&=\int_H^K \big(h^3(s)-C_d h(s)\big)(s-H)^{d-2}\,ds\\
&\le\bigg(1-\frac HK\bigg)^{d-2} \int_H^K \big(h^3(s)-C_d h(s)\big)s^{d-2}\,ds\,,
\end{split}\]
and analogously
\[\begin{split}
\int_{K-H}^{M-H} \big(\tilde h^3(s)- C_d \tilde h(s)\big)s^{d-2}\,ds
&=\int_K^M \big(h^3(s)- C_d h(s)\big)(s-H)^{d-2}\,ds\\
&\le\bigg(1-\frac HK\bigg)^{d-2} \int_K^M \big(h^3(s)-C_d h(s)\big)s^{d-2}\,ds\,.
\end{split}\]
The two last estimates together with~\eqref{removed} immediately imply~\eqref{thesis}, which complete this step.

As an immediate consequence of this step, to show the validity of~\eqref{goal} it is enough to consider the case $H=0$, and moreover once~\eqref{goal} will be proved it is already clear that an equality case can only happen with $H=0$. Hence, from now on we assume that $H=0$, hence $h$ is strictly decreasing.

\medskip\step{IV}{One has $h_-(M)=0$.}
This step is devoted to reduce ourselves to the case when $h_-(M)=0$. Note that, in this case, $h$ is continuous on the whole interval $[0,M]$. To be precise, we assume for a moment that $h_-(M)>0$, we call $M^+=M+(h_-(M)/h'_-(M))$, and we let $\tilde h:[0,M^+]\to[0,1]$ be the function given by
\[\tilde h(s)=\begin{cases}
h(s)&\hbox{if }0\le s<M\,, \\
h_-(M)-(s-M)h'_-(M)&\hbox{if }M\le s\le M^+\,.
\end{cases}\]

As in the preceding step, we will prove that if $h$ satisfies~\eqref{assmpt}, then
\be\label{obv}
\int_0^{M^+} \big(\tilde h^3(s)-C_d \tilde h(s)\big)s^{d-2}\,ds<0\,.
\ee
This is actually immediate. Indeed, since $h_-(M)<\sqrt{C_d}$ by Step~II, then
\[\int_0^{M^+} \big(\tilde h^3(s)-C_d \tilde h(s)\big)s^{d-2}\,ds<\int_0^M \big(h^3(s)-C_d h(s)\big)s^{d-2}\,ds\,,\]
and then the validity of~\eqref{obv} is obvious.

As a consequence of this step, we can from now on assume that $h_-(M)=0$, both for proving~\eqref{goal} and for checking the equality cases. Notice that, in particular, $h$ is a continuous, strictly decreasing, concave bijection of the interval $[0,M]$ onto the interval $[0,1]$. Therefore, $h'(s)$ is defined for almost every $s$ and a change of variable allows to rewrite the integral of~\eqref{goal} in the convenient form
\be\label{rewrite}
\int_0^M \big(h^3(s)-C_d h(s)\big)s^{d-2}\,ds=\int_0^1 \frac{(t^3-C_d t) (h^{-1}(t))^{d-2}}{|h'(h^{-1}(t))|}\,dt\,.
\ee

\medskip\step{V}{The function $h$ is affine in $[K,M]$ with $h'\equiv h_-'(K)$ on $(K,M)$.}
In this step we show that, in order to minimise the integral in~\eqref{goal}, the function $h$ must be affine in $[K,M]$. Let us be precise. We denote by $\tilde h:[0,M^+]\to[0,1]$ the largest positive concave function coinciding with $h$ on $[0,K]$, which is then an affine function on $[K,M^+]$ with $\tilde h'\equiv h_-'(K)$ there. We will show that
\be\label{goalstep5}
\int_0^M \big(h^3(s)-C_d h(s)\big)s^{d-2}\,ds\ge\int_0^{M^+} \big(\tilde h^3(s)-C_d \tilde h(s)\big)s^{d-2}\,ds\,,
\ee
with strict inequality unless $\tilde h=h$. As in the preceding steps this will imply that, in order to show~\eqref{goal} and to consider the equality cases, we can reduce ourselves to the case when $h$ is affine in $[K,M]$.

To show~\eqref{goalstep5}, we notice that for almost every $\sqrt{C_d}\le t\le1$ we have
\[\tilde h^{-1}(t)=h^{-1}(t)\qquad\hbox{and}\qquad\tilde h'(\tilde h^{-1}(t))=h'(h^{-1}(t))\,.\]
Instead, for almost every $0\le t\le\sqrt{C_d}$ we have
\[\tilde h^{-1}(t)\ge h^{-1}(t)\qquad\hbox{and}\qquad|\tilde h'(\tilde h^{-1}(t))|=|h_-'(K)|\le|h'(h^{-1}(t))|\,.\]
Therefore, \eqref{goalstep5} and its equality cases are a direct consequence of~\eqref{rewrite}.

\medskip\step{VI}{The function $h$ is affine in $[0,K]$.}
In this step we show that, to minimise the integral in~\eqref{goal}, it is convenient for the function $h$ to be affine also in the interval $[0,K]$. The argument is more delicate than in the preceding step, because there we could keep $h$ unchanged in $[0,K]$ and modify it only after $K$, while this time the modification of $h$ has necessarily effects both in $[0,K]$ and in $[K,M]$. To be precise, this time we define $\tilde h$ the affine function such that $\tilde h(0)=1$ and $\tilde h'\equiv h_-'(K)$. Notice that $\tilde h$ is defined in the interval $[0,M^-]$, being $M^-=|h_-'(K)|^{-1}$. As usual, we claim that if $h$ satisfies~\eqref{assmpt} and $\tilde h\ne h$, then
\be\label{goalstep6}
\int_0^{M^-} \big(\tilde h^3(s)-C_d \tilde h(s)\big)s^{d-2}\,ds<0\,.
\ee
To prove this inequality, we start by defining $K^-=K-(M-M^-)$. Notice that, in view of Step~V, $\tilde h(K^-)=h(K)=\sqrt{C_d}$, and for every $K\le s\le M$ we have $h(s)=\tilde h(s-(M-M^-))$. Therefore, for every $0\le t\le\sqrt{C_d}$ we have
\[\tilde h^{-1}(t) = h^{-1}(t)-(M-M^-)\,.\]
Hence, also recalling that for any such $t$ it is $t^3-C_dt\le0$, and by construction one has $h'(h^{-1}(t))=\tilde h'(\tilde h^{-1}(t))=h_-'(K)$, we get
\be\label{lowerinterval}
\int_0^{\sqrt{C_d}} \frac{(t^3-C_d t)(\tilde h^{-1}(t))^{d-2}}{|\tilde h'(\tilde h^{-1}(t))|}\,dt
\le\bigg(\frac{K^-}K\bigg)^{d-2}\int_0^{\sqrt{C_d}} \frac{(t^3-C_d t) (h^{-1}(t))^{d-2}}{|h'(h^{-1}(t))|}\,dt\,.
\ee
Notice that if $d>2$ then the strict inequality holds since $M^-<M$, which follows from the assumption that $\tilde h\ne h$. On the other hand, if $d=2$ then the equality necessarily holds.

Let us now consider the interval $[\sqrt{C_d},1]$. For any $t$ in this interval, by concavity of $h$ we have that $|h'(h^{-1}(t))|\leq |\tilde h'(\tilde h^{-1}(t)|=|h_-'(K)|$. Moreover,
\be\label{sublinear}
\tilde h^{-1}(t)<\frac{K^-}K\,h^{-1}(t)\qquad\forall\,t\in\big(\sqrt{C_d},1\big)\,.
\ee
In fact, this is an equality at $t=1$ and at $t=\sqrt{C_d}$, and then the strict inequality holds for every $t\in (\sqrt{C_d},1)$ since $h$ is concave and $\tilde h\neq h$ is affine. As a consequence, recalling that $t^3-C_dt\ge0$ in the interval, we obtain
\be\label{upperinterval}
\int_{\sqrt{C_d}}^1 \frac{(t^3-C_d t)(\tilde h^{-1}(t))^{d-2}}{|\tilde h'(\tilde h^{-1}(t))|}\,dt
<\bigg(\frac{K^-}K\bigg)^{d-2} \int_{\sqrt{C_d}}^1 \frac{(t^3-C_d t) (h^{-1}(t))^{d-2}}{|h'(h^{-1}(t))|}\,dt\,.
\ee
Notice that this time the strict inequality holds whatever $d$ is. In fact, if $d>2$ this comes from~\eqref{sublinear}, but also in the case $d=2$ the equality could be true only if $|h'(h^{-1}(t))|=|\tilde h'(\tilde h^{-1}(t)|$ for every $t\in(\sqrt{C_d},1)$, which is impossible since we are assuming $\tilde h\ne h$.

Recalling again~\eqref{rewrite}, we can now simply put together~\eqref{lowerinterval} and~\eqref{upperinterval} to deduce that, if~\eqref{assmpt} holds, then~\eqref{goalstep6} is true.

\medskip\step{VII}{Conclusion for the radial case.}
Putting together the preceding steps, we readily conclude the proof for the radial case. In fact, the arguments of Steps~I--VI ensure that, if there is a concave function $h:[0,M]\to[0,1]$ with $h(0)=1$ such that~\eqref{assmpt} holds, then there is another concave function, more precisely a decreasing, affine bijection $\tilde h:[0,\widetilde M]\to [0,1]$, for which~\eqref{assmpt} also holds true, with strict inequality unless $h$ is already a decreasing, affine bijection. On the other hand, a simple calculation ensures that for every decreasing, affine bijection $\tilde h:[0,\widetilde M]\to[0,1]$ inequality~\eqref{assmpt} holds as an equality. As a consequence, we deduce that the strict inequality in~\eqref{assmpt} never holds, and the equality holds if and only if $h$ is a decreasing, affine bijection of some interval $[0,M]$ onto $[0,1]$. Equivalently, inequality~\eqref{goal} is proved, and the equality cases are precisely the decreasing, affine bijections of some interval $[0,M]$ onto $[0,1]$.

\medskip\step{VIII}{Conclusion for the general case.}
We can now conclude the proof. Let $h:A\to[0,1]$ be as in the claim. The right inequality in~\eqref{thingendim}, together with its equality cases, is obvious. We now concentrate ourselves on the left inequality.

Let $h^*:B\to [0,1]$ be the radially symmetric decreasing rearrangement of $h$, defined on the $(d-1)$-dimensional ball $B$ centered at the origin and with the same area as $A$. The standard properties of the rearrangement imply that
\be\label{itsthesame}
\int_B (h^*)^3(x)\,dx=\int_A h^3(x)\,dx\,,\qquad\int_B h^*(x)\,dx=\int_A h(x)\,dx\,,
\ee
and moreover it is well-known that $h^*$ is also concave. Indeed, for every $t\ge0$ let us call
\[A_t=\big\{x\in A\ :\ h(x)>t\big\}\,.\]
Concavity of $h$ immediately yields that for every $t_1,\,t_2\ge0$ and every $\sigma\in[0,1]$ one has
\[\sigma A_{t_1} + (1-\sigma) A_{t_2} \subseteq A_{\sigma t_1+(1-\sigma)t_2}\,.\]
Brunn-Minkowski inequality gives then
\[
\big|A_{\sigma t_1+(1-\sigma)t_2}\big|^{1/d}\ge\big|\sigma A_{t_1}+(1-\sigma)A_{t_2}\big|^{1/d}
\ge\sigma\big|A_{t_1}\big|^{1/d}+(1-\sigma)\big|A_{t_2}\big|^{1/d}\,,
\]
and this last inequality is precisely the concavity of $h^*$. Since $h^*$ is concave, by Step~VII we know that~\eqref{goal} holds true, and by~\eqref{itsthesame} this means that the left inequality in~\eqref{thingendim} is also true. To conclude, we only have to study the equality cases in the left inequality in~\eqref{thingendim}.

Let then $h:A\to [0,1]$ be a function as in the claim for which the left inequality in~\eqref{thingendim} holds as an equality. As a consequence, also~\eqref{goal} holds for $h^*$ as an equality, thus the $1$-dimensional shape of $h^*$ is a decreasing, affine bijection of some interval $[0,M]$ onto $[0,1]$. Let $P$ be any point inside $A$ such that $h(P)=1$. Then, the cone function $h_C$ (in the sense of Definition~\ref{defconefct}) corresponding to the set $A$ and the point $P$ is smaller than $h$ by definition. On the other hand, (the $1$-dimensional shape of) its radially symmetric decreasing rearrangement $h_C^*$ is a decreasing, affine bijection of $[0,M]$ onto $[0,1]$, so we deduce $h^*=h_C^*$. And finally, as we already observed that $h_C\leq h$, we obtain that $h_C=h$, hence $h$ is in fact a cone function. Since, conversely, the left inequality in~\eqref{thingendim} clearly holds as an equality for every cone function, the proof is concluded.
\end{proof}

%

\section{The case $d=1$\label{sdone}}

The case $d=1$ allows more explicit calculations, since any open set $\O\subset\R^1$ is the union of disjoint open intervals $\O_k$. For each such interval we have
\begin{align*}
\lambda(\O_k)=\frac{\pi^2}{|\O_k|^2}\,, && T(\O_k)=\frac{|\O_k|^3}{12}\,,
\end{align*}
so that
\[
F_q(\O)=\frac{\pi^2}{12^q}\frac{\big(\sum_k|\O_k|^3\big)^q}{\big(\max_k|\O_k|\big)^2\big(\sum_k|\O_k|\big)^{3q-2}}\,.
\]
Setting $a_k=|\O_k|$ and denoting by $a$ the sequence $(a_k)_k$ we are reduced to study the quantity
\[
G_q(a)=\frac{\big(\sum_k a_k^3\big)^q}{\big(\max_k a_k\big)^2\big(\sum_k a_k\big)^{3q-2}}\,.
\]
With no loss of generality we may fix $\|a\|_\infty=a_1=1$, so $a_k\le1$ for all $k\geq 2$ and we may write the quantity above as
\[
G_q(a)=\frac{\big(1+\sum_{k\ge2}a_k^3\big)^q}{\big(1+\sum_{k\ge2}a_k\big)^{3q-2}}\,.
\]

Some easy calculations give the following bounds.

\begin{align*}
\hbox{If $q\le2/3$\,,}
&&
\left\{\begin{array}{ll}
\min G_q(a)=1\,,&\hbox{reached with $a_k=0$ for all $k\ge2$\,;}\\
\sup G_q(a)=+\infty\,,&\begin{aligned}&\hbox{asymptotically reached with $a_k=1$}\\
&\hbox{for all $2\le k\le N$, and $N\gg1$\,.}\end{aligned}
\end{array}\right.
\end{align*}

\begin{align*}
\hbox{If $2/3<q<1$\,,}
&&
\left\{\begin{array}{ll}
\inf G_q(a)=0\,,&\begin{aligned}&\hbox{asymptotically reached with $a_k=\eps$}\\
&\hbox{$\forall\ 2\le k\le N$, and $\eps^{-1}\ll N\ll\eps^{-3}$\,;}\end{aligned}\\
\sup G_q(a)=+\infty\,,&\begin{aligned}&\hbox{asymptotically reached with $a_k=1$}\\
&\hbox{for all $2\le k\le N$, and $N\gg1$\,.}\end{aligned}
\end{array}\right.
\end{align*}

\begin{align*}
\hbox{If $q\geq 1$\,,}
&&
\left\{\begin{array}{ll}
\inf G_q(a)=0\,,&\begin{aligned}&\hbox{asymptotically reached with $a_k=\eps$}\\
&\hbox{$\forall\ 2\le k\le N$, and $\eps^{-1}\ll N\ll\eps^{-3}$\,;}\end{aligned}\\
\max G_q(a)=1\,,&\hbox{reached with $a_k=0$ for all $k\geq 2$}\,.
\end{array}\right.
\end{align*}

In order to draw the Blaschke-Santal\'o diagram for the case $d=1$ we need the following lemma.

\begin{lemm}\label{realnum}
Let $p>1$ and let $A,B$ be nonnegative real numbers; then the following conditions are equivalent.
\begin{itemize}
\item[(a)]There exist $a_k\in[0,1]$ such that
\begin{align*}
A=\sum_ka_k^p\,, && B=\sum_ka_k\,.
\end{align*}
\item[(b)]The numbers $A,B$ satisfy the inequality
\[
A\le[B]+\big(B-[B]\big)^p
\]
where $[B]$ denotes the integer part of $B$.
\end{itemize}
\end{lemm}

\begin{proof}
$(a)\Rightarrow(b).$\ If $x,y\in[0,1]$ it is easy to see that
\[x^p+y^p\le\begin{cases}
(x+y)^p&\hbox{if }x+y\leq 1\,,\\
1+(x+y-1)^p&\hbox{if }x+y\geq 1.
\end{cases}\]
In other words, for a fixed sum $x+y$ with $x,y\in[0,1]$, we increase the quantity $x^p+y^p$ by replacing either the smallest one between $x$ and $y$ by $0$ or the bigger one by $1$. Repeating this argument for every pair of elements of the sequence $(a_k)$ we obtain what required.

$(b)\Rightarrow(a).$\ Taking $a_k=B/N$ for all $k=1,\dots,N$ we have
\begin{align*}
\sum_k a_k^p=\frac{B^p}{N^{p-1}}\,, &&\sum_k a_k=B\,.
\end{align*}
On the other hand, taking $N=[B]+1$ and
\[a_k=\begin{cases}
1&\hbox{if }k<N\\
B-[B]&\hbox{if }k=N
\end{cases}\]
we have
\begin{align*}
\sum_k a_k^p=[B]+\big(B-[B]\big)^p\,, &&\sum_k a_k=B\,.
\end{align*}
In this way, for a fixed $B=\sum_ka_k$ we can make $\sum_ka_k^p$ either arbitrarily small or equal to the bound $[B]+\big(B-[B]\big)^p$. A continuity argument concludes the proof.
\end{proof}

\section{The Blaschke-Santal\'o diagrams\label{sdiag}}

The Blaschke-Santal\'o diagram for $\lambda(\O)$ and $T(\O)$ consists in plotting the subset $E$ of $\R^2$ whose coordinates are determined by $\lambda(\O)$ and $T(\O)$ for some $\O$. In order to have variables $x$ and $y$ in the interval $[0,1]$, it is convenient to normalise the coordinates by
\begin{align}\label{defxy}
x=\frac{|B_1|^{2/d}\lambda(B_1)}{|\O|^{2/d}\lambda(\O)}\,, && y=\frac{|B_1|^{(d+2)/d}T(\O)}{|\O|^{(d+2)/d}T(B_1)}\,.
\end{align}
By scale invariance $B_1$ could be replaced by any ball $B$.

In the case $d=1$ we can give the full description of the Blaschke-Santal\'o diagram.

\begin{figure}[bp]
\centering{\includegraphics[scale=1.3]{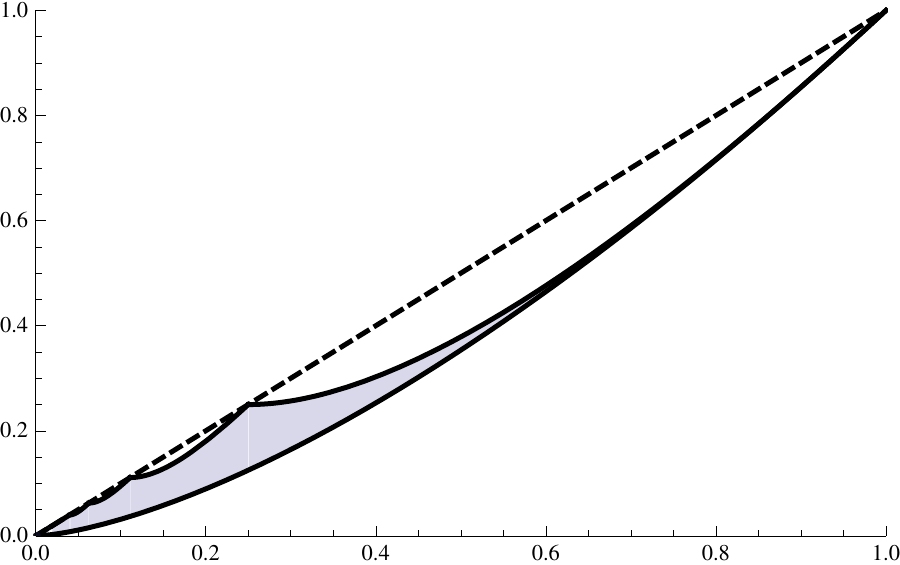}}
\caption{The Blaschke-Santal\'o diagram for $\lambda(\O)$ and $T(\O)$ in the case $d=1$.}\label{fig1}
\end{figure}

\begin{prop}\label{bsdim1}
Let $E$ be the Blaschke-Santal\'o diagram for $\lambda(\O)$ and $T(\O)$ in the case $d=1$. Then $(x,y)\in E$ if and only if
\[
x^{3/2}\le y\le x^{3/2}\Big(\big[x^{-1/2}]+\big(x^{-1/2}-\big[x^{-1/2}])^3\Big)\,.
\]
\end{prop}
\begin{proof}
Indeed, every $\O$ is the union of disjoint intervals $\O_k$ and, setting $a_k=|\O_k|$, we find
\begin{align*}
x=\frac{(\max_k a_k)^2}{\big(\sum_ka_k\big)^2}\,, && y=\frac{\sum_ka_k^3}{\big(\sum_ka_k\big)^3}\,.
\end{align*}
Since the quantities above are scaling invariant, we may choose $a_1=\max_k a_k=1$ and obtain
\begin{align*}
\left\{\begin{aligned}
x&=(1+B)^{-2}\\
y&=(1+A)(1+B)^{-3}\,,
\end{aligned}
\right.
&&
\hbox{where}
&&
\left\{
\begin{array}{l}
A=\sum\nolimits_{k\ge2}a_k^3\\[5pt]
B=\sum\nolimits_{k\ge2}a_k\,.
\end{array}
\right.
\end{align*}
The conclusion now follows by applying Lemma~\ref{realnum} with $p=3$.
\end{proof}

In the case of a dimension $d>1$ we can only provide some bounds to the Blaschke-Santal\'o diagram $E$.

\begin{prop}\label{bsdimd}
Let $E$ be the Blaschke-Santal\'o diagram for $\lambda(\O)$ and $T(\O)$ in the case of dimension $d$. Then
\[
\Big\{ x^{(d+2)/2}\le y\le x^{(d+2)/2}\Big(\big[x^{-d/2}]+\big(x^{-d/2}-\big[x^{-d/2}])^{(d+2)/d}\Big)\Big\}
\subset E\subset
\Big\{x^{(d+2)/2}\le y\Big\} \,.
\]
\end{prop}

\begin{proof}
Considering domains $\O$ which are union of disjoint balls $\O_k$ with radii $r_k$, and arguing as above, we find:
\begin{align*}
x=\frac{(\max_k r_k)^2}{\big(\sum_k r_k^d\big)^{2/d}}\,, && y=\frac{\sum_k r_k^{d+2}}{\big(\sum_k r_k^d\big)^{(d+2)/d}}\,.
\end{align*}
Using again the scaling invariance, we may choose $r_1=\max_k r_k=1$ and, setting $a_k=r_k^d$, we obtain
\begin{align*}
\left\{\begin{aligned}
x&=(1+B)^{-2/d}\\
y&=(1+A)(1+B)^{-(d+2)/d}\,,
\end{aligned}
\right.
&&
\hbox{where}
&&
\left\{
\begin{array}{l}
A=\sum\nolimits_{k\geq 2}a_k^{(d+2)/d}\\[5pt]
B=\sum\nolimits_{k\geq2}a_k\,.
\end{array}
\right.
\end{align*}
The left inclusion now follows by applying Lemma~\ref{realnum} with $p=(d+2)/d$. The right inclusion, instead, is equivalent to say that $T(\Omega)\lambda(\Omega)^{(d+2)/2}$ is minimized by the ball, which is the Kohler-Jobin inequality~(\ref{kjineq}).
\end{proof}

\begin{figure}[htbp]
\centering{\includegraphics[scale=1.3]{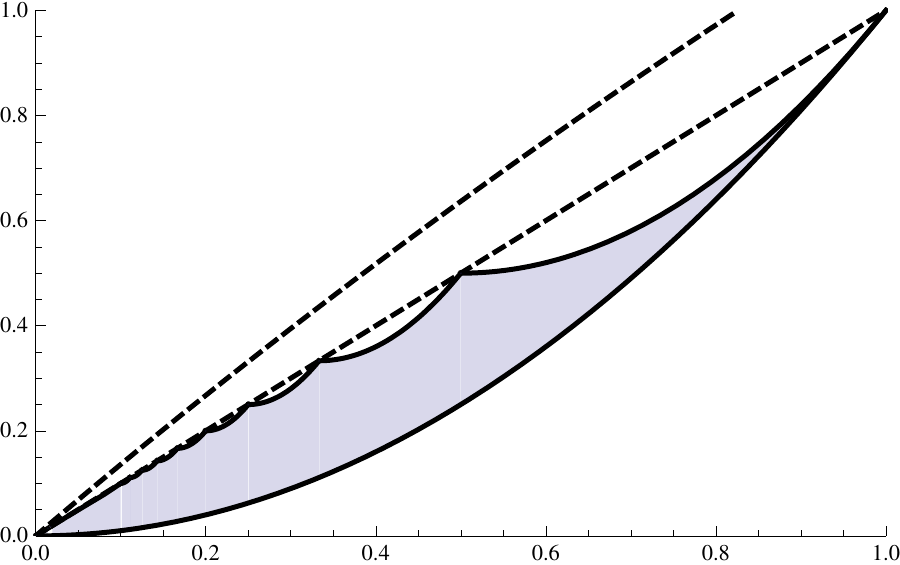}}
\caption{The colored region is the lower bound of the Blaschke-Santal\'o diagram $E$ for $\lambda(\O)$ and $T(\O)$ in the case $d=2$. The upper dashed line is the upper bound for the set $E$ given by~(\ref{donethere}).}\label{fig2}
\end{figure}

\section{Further remarks and open questions\label{sconc}}

The optimisation problems we presented are very rich, and several questions remain open. below we list some of them, together with some comments.

\bigskip

{\bf Problem 1. }We have seen that when $q>1$
\[\sup\Big\{F_q(\O)\ :\ \O\textup{ open in }\R^d,\ |\O|<\infty\Big\}<+\infty\,.\]
It would be interesting to establish if the supremum above is actually a maximum or if it is only achieved asymptotically by a maximising sequence $\O_n$. Note that the class of competing domains is the whole class of open sets in $\R^d$ with finite measure, without any other geometric or topological restriction.

\bigskip

{\bf Problem 2. }When $q\le2/(d+2)$ the Kohler-Jobin inequality~\eqref{kjineq} implies that
\[\min\Big\{F_q(\O)\ :\ \O\textup{ open in }\R^d,\ |\O|<\infty\Big\}=F_q(B),\]
where $B$ is any ball in $\R^d$. It would be interesting to prove or disprove a kind of {\it reverse Kohler-Jobin inequality}, that is the existence of another threshold $Q_d>1$ such that for $q\ge Q_d$
\[\max\Big\{F_q(\O)\ :\ \O\textup{ open in }\R^d,\ |\O|<\infty\Big\}=F_q(B)\,.\]
It is easy to see that if the ball $B$ maximises the shape functional $F_q$ for a certain $q>1$, then $B$ also maximises $F_p$ for every $p>q$. Indeed, we have
\[F_p(\O)=F_q(\O)\Big(\frac{T(\O)}{|\O|^{(d+2)/d}}\Big)^{p-q}\le F_q(B)\Big(\frac{T(B)}{|B|^{(d+2)/d}}\Big)^{p-q}=F_p(B)\,,\]
where we have used the fact that the quantity $T(\O)\,|\O|^{-(d+2)/d}$ is maximal for $\O=B$.

\bigskip

{\bf Problem 3. }
From Figure~\ref{fig2} we see that the Blaschke-Santal\'o set $E$ is bounded from below by the Kohler-Jobin line with equation $y=x^{(d+2)/2}$, and every point on this line can be asymptotically reached by a sequence of domains $\O_n$ made by the union of disjoint balls. On the contrary, the upper bound of $E$ is less clear: we believe that a continuous curve of equation $y=S(x)$ should exist such that
\[\overline E=\big\{(x,y)\in\R^2\ :\ x\in[0,1],\ x^{(d+2)/2}\le y\le S(x)\big\}\,.\]
The proof of this fact is at the moment missing and would require that the set $E$ is {\it convex horizontally} and {\it convex vertically} in the sense that the intersections of $E$ with horizontal and vertical straight lines are segments.

\bigskip

{\bf Problem 4. }It is not clear what may happen if we consider intermediate classes of domains as for instance in the case $d=2$
\[\A_{sc}=\big\{\O\hbox{ simply connected }\big\}\,,\]
or more generally
\[\A_{sc,N}=\big\{\O\hbox{ with a topological genus not exceeding }N\big\}\,.\]
For $q\le2/(2+d)$ we still have that a ball is a minimizer of $F_q$, while for $q>2/(2+d)$ the infimum of $F_q$ is still zero, asymptotically reached as in Proposition~\ref{finf}, eventually connecting the balls by means of very thin channels. The situation for the sup is less clear. We obviously have
\[C_d^+\le\sup\Big\{F_1(\O)\ :\ \O\in\A_{sc}\hbox{ in $\R^d$}\Big\}\le1\]
and it would be interesting to see if some of the inequalities above are strict. Similar questions arise (for the sup as well as for the inf, and in any space dimension) if we consider even smaller classes of domains as
\[\A_{ss}=\big\{\O\hbox{ star-shaped}\big\}\quad\hbox{in any dimension }d\,.\]

\section{Appendix: The case $q=1$\label{sq1}}

The inequality
\begin{align*}
0<F_1(\O)<1 && \hbox{for all }\O\subset\R^d
\end{align*}
was first proven by P\'olya and Szeg\"o and can be found in~\cite{ps51} (see also Proposition 2.3 of~\cite{bbv15}). In~\cite{bfnt16} the improvement
\[
F_1(\O)\le1-\frac{2d\omega^{2/d}_d}{d+2}\frac{T(\O)}{|\O|^{1+2/d}}
\]
is proved, together with the fact that the supremum of $F_1$ is actually 1. We give here a quick proof that uses the well known fact that if $d\ge2$ the closure under $\Gamma$-convergence of the Dirichlet energies $\int_\O|\nabla u|^2\,dx$ defined on $H^1_0(\O)$, with $\O\subset D$, with $D$ a fixed bounded Lipschitz domain of $\R^d$, consists of all functionals of the form
\[\int_D|\nabla u|^2\,dx+\int_D u^2\,d\mu\]
where $\mu$ runs among all {\it capacitary measures} on $D$, that is nonnegative Borel measures, possibly taking the value $+\infty$, that vanish on all sets of capacity zero (see for instance~\cite{bubu05}). In particular, all the measures of the form $c\,dx$ with $c>0$ can be reached by limits of domains $\O_n$. In addition, both the eigenvalues and the torsional rigidity are continuous for the convergence above.

In order to prove that the supremum of $F_1$ is $1$ is then sufficient to show that
\[
\sup_{c,D}\bigg[\frac{\lambda_c(D)T_c(D)}{|D|}\bigg]\ge1\,,
\]
where
\[\begin{split}
&\lambda_c(D)=\min\bigg\{\int_D|\nabla u|^2\,dx+c\int_D u^2\,dx\ :\ u\in H^1_0(D),\ \int_D u^2\,dx=1\bigg\},\\
&T_c(D)=\max\bigg\{\bigg[\int_D u\,dx\bigg]^2\bigg[\int_D|\nabla u|^2\,dx+c\int_D u^2\,dx\bigg]^{-1}\ :\ u\in H^1_0(\O)\setminus\{0\}\bigg\}\,.
\end{split}\]
We have immediately $\lambda_c(D)=c+\lambda(D)$. In order to estimate $T_c(D)$ from below, let $D$ be unit ball in $\R^d$ centered at the origin, let $\delta>0$ be fixed, and let $u_\delta\in H^1_0(D)$ be the function
\[u_\delta(x)=
\begin{cases}
1&\hbox{if }|x|\le 1-\delta,\\
(1-|x|)/\delta&\hbox{if }|x|>1-\delta\,.
\end{cases}\]
Then
\[\begin{split}
T_c(D)
&\geq\Big[\int_D u_\delta\,dx\Big]^2\Big[\int_D|\nabla u_\delta|^2\,dx+c\int_D u_\delta^2\,dx\Big]^{-1}\\
&\geq\Big[\omega_d(1-\delta)^d\Big]^2\Big[\delta^{-2}\omega_d\big(1-(1-\delta)^d\big)+c\omega_d\Big]^{-1}
\geq\omega_d(1-\delta)^{2d}\Big[\delta^{-2}+c\Big]^{-1}\,,
\end{split}\]
so that
\[
\frac{\lambda_c(D)T_c(D)}{|D|}\ge\frac{(c+\lambda(D))(1-\delta)^{2d}}{\delta^{-2}+c}\,.
\]
As $c\to+\infty$ we obtain
\[
\sup_c\bigg[\frac{\lambda_c(D)T_c(D)}{|D|}\bigg]\ge(1-\delta)^{2d}\,.
\]
Finally, letting $\delta\to0$ we have what was to be proved.

\bigskip

\noindent{\bf Acknowledgments.} MvdB was supported by The Leverhulme Trust through Emeritus Fellowship EM-2018-011-9. The work of GB and AP is part of the project 2017TEXA3H {\it``Gradient flows, Optimal Transport and Metric Measure Structures''} funded by the Italian Ministry of Research and University.

\bigskip
{\small\noindent
Michiel van den Berg:
School of Mathematics, University of Bristol\\
Fry Building, Woodland Road\\
Bristol BS8 1UG - UK\\
{\tt M.vandenBerg@bristol.ac.uk}\\
{\tt http://www.maths.bris.ac.uk/~mamvdb/}

\bigskip\noindent
Giuseppe Buttazzo:
Dipartimento di Matematica,
Universit\`a di Pisa\\
Largo B. Pontecorvo 5,
56127 Pisa - ITALY\\
{\tt giuseppe.buttazzo@dm.unipi.it}\\
{\tt http://www.dm.unipi.it/pages/buttazzo/}

\bigskip\noindent
Aldo Pratelli:
Dipartimento di Matematica,
Universit\`a di Pisa\\
Largo B. Pontecorvo 5,
56127 Pisa - ITALY\\
{\tt aldo.pratelli@dm.unipi.it}\\
{\tt http://pagine.dm.unipi.it/pratelli/}}

\end{document}